\documentclass{amsart}

\usepackage{latexsym}
\usepackage{amsfonts,amssymb} 
\usepackage{graphicx}
\usepackage{tikz}
\usepackage{caption}
\usepackage{subcaption}

\newcommand{\field}[1]{\mathbb{#1}}
\newcommand{\A}{\field{A}}
\newcommand{\C}{\field{C}}
\newcommand{\D}{\field{D}}

\newcommand{\HH}{\field{H}}
\newcommand{\N}{\field{N}}
\newcommand{\PP}{\field{P}}

\newcommand{\R}{\field{R}}
\newcommand{\Sp}{\field{S}}

\newcommand{\Z}{\field{Z}}

\theoremstyle{plain}

\newtheorem{theorem}{Theorem}[section]
\newtheorem{proposition}[theorem]{Proposition}
\newtheorem{lemma}[theorem]{Lemma}
\newtheorem{corollary}[theorem]{Corollary}
\newtheorem{definition}[theorem]{Definition}
\newtheorem{remark}[theorem]{Remark}

\newtheorem{example}[theorem]{Example}

\theoremstyle{definition}

\theoremstyle{remark}

\begin{document}

% Redefine "plain" pagestyle
\makeatletter	   % `@' is now a normal "letter' for LaTeX
\makeatother     % `@' is restored as a "non-letter" character

\title{The Polar Group of a Real Form of an Affine or Projective $\C$-Variety}
\author{Gene Freudenburg}
\date{\today} 

\maketitle

\pagestyle{plain}

\begin{abstract} A general problem is to classify the real forms of a complex variety up to isomorphism. 
This paper introduces the polar group of a real form $X$ of a complex variety $Y$ as a tool to distinguish such real forms. This group is an invariant of $X$ which encodes information about the residual divisors in the coordinate ring of $Y$ over the coordinate ring of $X$. We calculate polar groups for various curves, including the real line $\R^1$ and the algebraic 1-sphere $\Sp^1$.
\end{abstract}

\section{Introduction}

A real form of an integral $\C$-domain $B$ is any integral $\R$-domain $A$ such that $\R$ is algebraically closed in $A$ and $B=\C\otimes_{\R}A=A[i]$. 
A general problem is to classify the real forms of $B$ up to isomorphism. 
One birational invariant associated to a real form is the abelian group $L^*/K^*$, where $K$ is the field of fractions of $A$ with unit group $K^*$, 
and $L=K[i]$ is the field of fractions of $B$ with unit group $L^*$. A key feature of this group is that $[c]^{-1}=[\bar{c}]$, where $[c]$ denotes the class of $c\in L^*$, and $\bar{c}$ is the conjugate of $c$. However, this invariant will not distinguish between non-isomorphic real forms which are birationally isomorphic. 
In this article, we study the quotient group $L^*/B^*K^*$, where $B^*$ is the unit group of $B$. This group is an invariant of the ring $A$ which encodes information about the residual divisors in $B$ over $A$.
We call it the {\bf polar group} of $A$, denoted $\Pi (A)$. 

Affine rings are of particular interest. If $A$ is finitely generated over $\R$, then ${\rm Spec}(A)$ is a real form of ${\rm Spec}(B)$ and we define $\Pi ({\rm Spec}(A))=\Pi (A)$. 
If $A$ is $\Z$-graded, then the grading extends to $B$ and ${\rm Proj}(A)$ is a real form of ${\rm Proj}(B)$. We define $\Pi ({\rm Proj}(A))=L_0^*/B_0^*K_0^*$, where $L_0\subset L$ and $K_0\subset K$ are the subfields of elements of degree zero and $B_0=B\cap L_0$. 

An affine variety (over any ground field) is factorial if its coordinate ring is a unique factorization domain. Suppose that $X$ is an affine $\R$-variety and $Y=\C\otimes X$. 
It can happen that $X$ is factorial and $Y$ is not factorial; that $X$ is not factorial and $Y$ is factorial; or that both $X$ and $Y$ are factorial. 
The algebraic spheres $\Sp^n$ serve to illustrate these properties. $\Sp^1$ is a non-factorial curve, while its complexification $\C^*$ is factorial. $\Sp^2$ is a factorial surface, 
while the complex surface $\C\otimes\Sp^2$ is not factorial. $\Sp^3$ and its complexification $SL_2(\C)$ are both factorial. 

The strongest results about polar groups are obtained when $X$ or $Y$ is factorial. If $X$ is factorial, then $\Pi (X)$ is torsion free ({\it Thm.\,\ref{A=UFD3}}). If $Y$ is factorial, 
{\it Thm.\,\ref{free}} gives a criterion for $\Pi (X)$ to be free abelian, and gives a basis when this criterion is met. 
In particular, $\Pi (X)$ is free abelian if both $X$ and $Y$ are factorial ({\it Cor.\,\ref{A=UFD2}}).

{\it Section 6} computes polar groups for various curves, including $\Sp^1$. As an application, polar groups are used in \cite{Freudenburg.ppt18b} to prove that there is only one polynomial embedding of $\Sp^1$ in $\R^2$, up to the action of a plane automorphism. 

\subsection{Notation} 
\begin{enumerate}
\item The group of units of the ring $R$ is denoted $R^*$. 
\item The multiplicative monoid $R\setminus\{ 0\}$ of the ring $R$ is denoted $R'$. 
\item The polynomial ring in $n$ variables over the ring $R$ is denoted $R^{[n]}$.
\item Let $k=\R$ or $\C$. The automorphism group of $B$ as a $k$-algebra is ${\rm Aut}_k(B)$, 
affine $n$-space over $k$ is $\A^n_k$, and projective $n$-space over $k$ is $\PP^n_k$.
\item The real $n$-sphere $\Sp^n$ is the algebraic $\R$-variety with coordinate ring:
\[
\R[x_0,\hdots ,x_n]/(x_0^2+\cdots +x_n^2-1)
\]
\item Given $z=a+bi\in\C$ for $a,b\in\R$, $\| z\|=\sqrt{a^2+b^2}$.
\item The closed unit disk is $\D =\{ z\in\C \, |\, \| z\|\le 1\}$.
\item The upper half plane of the complex plane is $\HH =\{ z\in\C \, |\, {\rm Im}(z)>0\}$.
\end{enumerate}

%%%%%%%%%%%%%%%%%%%%%%%%%%%%%%%%%%%%%%%%%%%%%%%%%%%%%%%%%%%%%%%%%%%%%%%%%%%%%%

\section{Conjugations}\label{conj}

\subsection{Automorphisms} Consider any integral domain $B$ containing $\C$. 
Assume that ${\rm Gal}(\C/\R )\cong\Z_2$ acts on $B$ by $\R$-algebra automorphisms and that the inclusion $\C\to B$ equivariant. 
This is equivalent to the condition that $B$ admits a real form.
${\rm Gal}(\C/\R )$ acts on ${\rm Aut}_{\C}(B)$ by group-theoretic conjugation, and if $A=B^{{\rm Gal}(\C /\R )}$, then:
\[
{\rm Aut}_{\R}(B)={\rm Aut}_{\C}(B)\rtimes {\rm Gal}(\C/\R) \quad {\rm and}\quad {\rm Aut}_{\R}(A)={\rm Aut}_{\C}(B)^{{\rm Gal}(\C /\R )}
\]
Let $\langle\mu\rangle$ be the image of ${\rm Gal}(\C /\R )$ in ${\rm Aut}_{\R}(B)$. 

\begin{definition} A {\bf conjugation} of $B$ is an involution in the coset ${\rm Aut}_{\C}(B)\mu$, i.e., an $\R$-algebra involution of $B$ which restricts to complex conjugation on $\C$. \end{definition}
Let ${\rm Conj}(B)$ be the set of all conjugations of $B$.
Given $\alpha\in {\rm Aut}_{\C}(B)$, define $\bar{\alpha}=\mu\alpha\mu$, noting that 
$\overline{\alpha\beta}=\bar{\alpha}\bar{\beta}$. Then:
\[
\alpha\mu\in {\rm Conj}(B) \quad\Leftrightarrow\quad  \alpha\bar{\alpha}=1
\]
Note that the product of two elements of ${\rm Conj}(B)$ is an element of ${\rm Aut}_{\C}(B)$, and if $\alpha\in {\rm Aut}_{\C}(B)$ is an involution which commutes with $\sigma\in {\rm Conj}(B)$, then $\alpha\sigma\in{\rm Conj}(B)$. 

Given $\sigma\in {\rm Conj}(B)$, let $A=B^{\sigma}$, the fixed ring of $\sigma$. Then $\R\subset A$ and $\R$ is algebraically closed in $A$ (since $A\cap\C=\R)$. 
In addition, $B$ is an $A$-module 
and $\sigma$ is an $A$-module homomorphism. Given $f\in B$, note that $f+\sigma (f), i(f-\sigma (f))\in A$. 
Define:
\[
\textstyle f_1=\frac{1}{2}(f+\sigma (f)) \quad {\rm and}\quad f_2=-\frac{1}{2}i(f-\sigma (f))
\]
Then $f_1,f_2\in A$ and $f=f_1+if_2$. 
Consequently, $B=A\oplus iA$ and $A$ is a real form of $B$.

Let $G={\rm Aut}_{\C}(B)$ and define the mapping ${\rm adj}:G\to G$ by ${\rm adj}(g)=[g,\mu]$. Then ${\rm adj}(g)\overline{{\rm adj}(g)}=1$ for every $g\in G$. So there is an induced mapping $G\to {\rm Conj}(B)$ given by $g\to {\rm adj}(g)\mu = g\mu g^{-1}$, thus mapping $G$ onto the conjugacy class of $\mu$ in ${\rm Conj}(B)$. 

\subsection{Affine and Projective Varieties}
\begin{proposition}\label{finite} If $B=\C [x_1,...,x_n]$ and $\sigma\in {\rm Conj}(B)$, then:
\[
B^{\sigma}=\R [x_1+\sigma (x_1),\hdots ,x_n+\sigma (x_n) , i(x_1-\sigma (x_1)),\hdots , i(x_n-\sigma (x_n))]
\]
\end{proposition}

\begin{proof} Let $R\subset B^{\sigma}$ denote the subring:
\[
R=\R [x_1+\sigma (x_1),\hdots ,x_n+\sigma (x_n) , i(x_1-\sigma (x_1)),\hdots , i(x_n-\sigma (x_n))]
\]
We have $B^{\sigma}=\R [b+\sigma(b), i(b-\sigma (b))\,\vert\, b\in B]$. Given $N\in\N$, define:
\[
M_N=\{ x_1^{e_1}\cdots x_n^{e_n}\,\vert\, e_j\in\N, e_1+\cdots +e_n\le N\}
\]
Then $m+\sigma (m), i(m-\sigma (m))\in R$ for $m\in M_1$. Assume by induction that, for some $N\ge 1$, $m+\sigma (m), i(m-\sigma (m))\in R$ whenever $m\in M_N$. 
Given $m\in M_N$ and $j$, $1\le j\le n$, we have
\[
(x_j+\sigma (x_j))(m+\sigma (m))-(i)(x_j-\sigma (x_j))(i)(m-\sigma (m)) =2(x_jm+\sigma (x_jm))
\]
and:
\[
i(x_j-\sigma (x_j))(m+\sigma (m))+i(m-\sigma (m))(x_j+\sigma (x_j)) =2i(x_jm-\sigma (x_jm))
\]
It follows that $m'+\sigma (m'), i(m'-\sigma (m'))\in R$ whenever $m'\in M_{N+1}$. Therefore, $f+\sigma (f), i(f-\sigma (f))\in R$ for any monomial $f$ in $x_1,\hdots , x_n$. From this, we conclude that $R=B^{\sigma}$.
\end{proof}

Note that {\it Prop.\,\ref{finite}} shows that, if $B$ is affine over $\C$ and $\sigma\in {\rm Conj}(B)$, then $B^{\sigma}$ is affine over $\R$. 
More precisely, if $B$ is generated by $n$ elements over $\C$, then $A=B^{\sigma}$ is generated by $2n$ elements over $\R$. Quite often, we want to know if $A$ can be generated by $n$ elements. 

For affine $B$, let $Y={\rm Spec}(B)$ and $X={\rm Spec}(B^{\sigma})$. The induced map $\sigma^* :Y\to Y$ is a {\bf conjugation} of $Y$, and the fixed points of $\sigma^*$ (possibly empty) correspond to the $\R$-rational points of $X$.  

If $A$ is a graded ring, then ${\rm Proj}(A)$ is a real form of ${\rm Proj}(B)$. More generally, conjugations and real forms can be defined for certain types of schemes over $\C$;
see \cite{Hartshorne.77}, Exercise II.4.7. 
Two important open questions are (1) whether the number isomorphism classes of real forms of a complex rational surface is always finite (see \cite{Benzerga.16}), and (2) whether an affine $\C$-variety can admit more than one compact real form. 

\subsection{Examples}

The following three examples are closely related. 

\begin{example} {\rm The field of rational functions $L=\C (t)$ has two real forms, namely, $K_1=\R (t)$ and $K_2=\R (x,y)$, where $x^2+y^2+1=0$. To show $K_1\not\cong K_2$, it suffices to observe that $K_1$ is a formally real field, but $K_2$ is not formally real; see \cite{Freudenburg.17}, Lemma 10.10. Let $\mu$ be standard conjugation on $L$, i.e., $\mu (z)=\bar{z}$ for $z\in\C$ and $\mu (t)=t$. 
Let $\beta$ be the $\C$-automorphism of $L$ defined by $\beta (t)=-t^{-1}$, and define a second conjugation of $L$ by $\sigma =\beta\mu$. 
Then $K_1=L^{\mu}=\R (t)$ and $K_2=L^{\sigma}=\R (t-t^{-1}, i(t+t^{-1}))$.}
\end{example}

\begin{example}\label{proj-line}  {\rm $\PP^n_{\R}$ is a real form of $\PP^n_{\C}$ for all $n\ge 0$. For even $n$, this is the only real form of projective space, but for odd $n$ there is a second. Consider the case $n=1$; see \cite{Benzerga.16} for the general case.

Let $X_{\R}$ be the conic in $\PP^2_{\R}$ defined by $x_0^2+x_1^2+x_2^2=0$, and let $X_{\C}$ be the conic in $\PP^2_{\C}$ defined by the same form. Then $X_{\R}$ is a real form of $X_{\C}$. If $\phi :\PP^2_{\C}\to\PP^2_{\C}$ is standard conjugation, then $\phi$ restricts to $X_{\C}$, so this restriction is a conjugation of $X_{\C}$, and ${\rm Fix}(\phi)\cap X_{\C}=\emptyset$. Since $X_{\C}\cong\PP^1_{\C}$, we see that $X_{\R}$ is a real form of $\PP^1_{\C}$ and $X_{\R}\not\cong\PP^1_{\R}$. 
The reader can check that, if $\psi (y_0;y_1)=(\bar{y}_0;\bar{y}_1)$ is standard conjugation on $\PP^1_{\C}$, then the non-standard conjugation is given by 
$\sigma=\alpha\psi$, where $\alpha$ is the $\C$-automorphism $\alpha (y_0;y_1)=(-y_1;y_0)$, i.e., $\sigma (y_0;y_1)=(-\bar{y}_1;\bar{y}_0)$. }
\end{example} 

\begin{example}\label{Laurent} {\rm Let $B=\C [t,t^{-1}]$ be the ring of Laurent polynomials. Then ${\rm Spec}(B)=\C^*$, the punctured complex line, and ${\rm Aut}_{\C}(B)=B^*\rtimes\Z_2$, where $B^*$  acts by multiplication and $\Z_2$ interchanges $t$ and $t^{-1}$. 
There are three inequivalent real forms of $\C^*$.
\begin{enumerate}
\item Define $\phi\in {\rm Conj}(B)$ by $\phi (z)=\bar{z}$ for $z\in\C$ and $\phi (t)=t$. Then:
\[
B^{\phi}=\R [t,t^{-1}]=\R [X,Y] \quad {\rm where}\quad  XY=1
\]
Geometrically, $\phi^* :\C^*\to \C^*$ is given by $z\to\bar{z}$, with fixed points ${\rm Fix}(\phi^*)=\R^*$.
\smallskip
\item Define $\alpha\in {\rm Aut}_{\C}(B)$ by $\alpha (t)=t^{-1}$ and $\psi\in {\rm Conj}(B)$ by $\psi =\alpha\phi$. Then:
\[
B^{\psi}=\R [t+t^{-1}, i(t-t^{-1})]=\R [X,Y] \quad {\rm where}\quad  X^2+Y^2=1
\]
Geometrically, $\psi^*$ is given by $z\to\bar{z}^{-1}$, with fixed points ${\rm Fix}(\psi^*)=\Sp^1$.
\smallskip
\item Define $\beta\in {\rm Aut}_{\C}(B)$ by $\beta (t)=-t^{-1}$ and $\sigma\in {\rm Conj}(B)$ by $\sigma =\beta\phi$. Then:
\[
B^{\sigma}=\R [t-t^{-1}, i(t+t^{-1})]=\R [X,Y] \quad {\rm where}\quad  X^2+Y^2=-1
\]
Geometrically, $\sigma^*$ is given by $z\to-\bar{z}^{-1}$, with fixed points ${\rm Fix}(\sigma^*)=\emptyset$. 
\end{enumerate}
}
\end{example}

%%%%%%%%%%%%%%%%%%%%%%%%%%%%%%%%%%%%%%%%%%%%%%%%%%%%%%%%%%%%%%%%%%%%%%%%%%%%%%

\section{Polar Factorization}
Hereafter, we assume that $A$ is a noetherian integral $\R$-domain in which $\R$ is algebraically closed, and that $B=\C\otimes_{\R}A$ is an integral domain. 
In this case, $B$ is also noetherian. Note that, if $B$ is normal, then $A$ is normal. In addition, recall:
\begin{center}
{\it If $B$ is a UFD and $B^*=\C^*$, then $A$ is a UFD.}
\end{center}

$B$ is an integral extension of $A$ of degree 2 given by $B=A[i]$. In this case, $B$ is an $A$-module of rank 2 given by $B=A\oplus iA$, and the $A$-module map sending $i$ to $-i$ is a conjugation.
If $K={\rm frac}(A)$ and $L={\rm frac}(B)$, then $L=K[i]$ and $K\cap B=A$.

Given $f,g\in B$, suppose that $f=f_1+if_2$ for $f_1,f_2\in A$. 
The {\bf conjugate} of $f$ is $\bar{f}=f_1-if_2$. For any set $S\subset B$, $\bar{S}=\{ \bar{s}\, |\, s\in S\}$.
The matrix form of $f$ is:
\[
N_f=\begin{pmatrix} f_1 & -f_2 \cr f_2 & f_1 \end{pmatrix}
\]
The minimal algebraic relation of $f$ over $A$ is given by the minimal polynomial of $N_f$. Note that $N_{\bar{f}}=N_f^T$.

\begin{definition} $f\in B'$ has {\bf no real divisor} if $f\in r\cdot B$ and $r\in A'$ implies $r\in A^*$.\end{definition}
Let $\Delta (B)$ be the set of elements of $B'$ with no real divisor. 

\begin{definition} A {\bf polar factorization} of $f\in B'$ is a factorization of the form $f=\alpha h$, $\alpha\in A'$, $h\in\Delta (B)$.\end{definition}

The group $\Gamma =B^*\rtimes\Z_2$ acts on $B'$, where $B^*$ acts by multiplication, and $\Z_2$ acts by $f\to\bar{f}$. Observe the following properties.
\begin{enumerate}
\item $B^*\subset\Delta (B)$ and $A\cap\Delta (B)=A^*$. 
\item  Given $f,g\in B'$, if $fg\in \Delta (B)$, then $f\in \Delta (B)$ and $g\in \Delta (B)$. 
\item The action of $\Gamma$ on $B'$ restricts to $\Delta (B)_1:=\{ f\in\Delta (B) \, |\, f \,\, {\rm is\,\, irreducible} \}\subset \Delta (B)$. 
\item $B'$ is generated as a monoid by $A'$, $B^*$ and $\Delta (B)_1$. 
\end{enumerate}

\begin{lemma}\label{ACCP1} Every element of $B'$ admits a polar factorization. 
\end{lemma}

\begin{proof} Given $g_0\in B'$, suppose that $g_0=r_1g_1$ for $r_1\in A'$ and $g_1\in B'$. Inductively, write 
$g_{n+1}=r_ng_n$ for $r_n\in A'$ and $g_n\in B'$. Then 
\[
g_0B\subseteq g_1B\subseteq g_2B\subseteq \cdots
\]
is an ascending chain of principal ideals. Since $B$ is noetherian, there exists $n\ge 0$ such that $g_{n+1}B=g_nB$. We thus have:
\[
r_{n+1}g_{n+1}=g_n \implies r_{n+1}g_{n+1}B=g_nB=g_{n+1}B \implies r_{n+1}\in B^*\cap A=A^*
\]
Therefore, $g_{n+1}\in\Delta (B)$, and if $r=r_1\cdots r_{n+1}$, then $g_0=rg_{n+1}$ is a polar factorization. 
\end{proof}

In an integral domain $R$, elements $f,g\in R$ are {\bf relatively prime} if $fR\cap gR=fgR$. When $R$ is a UFD, this coincides with the usual notion of relative primeness. 

\begin{lemma}\label{conj-gcd} Given $f\in B'$, if $f$ and $\bar{f}$ are relatively prime, then $f\in\Delta (B)$. 
\end{lemma}

\begin{proof} By {\it Lemma\,\ref{ACCP1}}, $f$ admits a polar factorization $f=\alpha h$, where $\alpha\in A'$ and $h\in\Delta (B)$. Then $\bar{f}=\alpha \bar{h}$. Therefore:
\[
\alpha h\bar{h}\in fB\cap\bar{f}B=f\bar{f}B=\alpha^2h\bar{h}B \implies \alpha\in B^*\cap A=A^*
\]
It follows that $f\in\Delta (B)$. 
\end{proof}

The following is a special case of {\it Lemma\,2.46} in \cite{Freudenburg.17}. 

\begin{lemma} Given $f,g\in A$, if $f$ and $g$ are relatively prime in $A$, then $f$ and $g$ are relatively prime in $B$.
\end{lemma}

For an abelian group $G$, the {\bf square-torsion subgroup} of $G$ is:
\[
\{ \gamma\in G\, |\, \gamma^2=1\}
\]
$G$ is {\bf square-torsion free} if its square-torsion subgroup is trivial. 

Suppose that $G$ is an abelian group which is torsion free. In general, $G$ need not be free abelian. However, if $G$ is finitely generated as a $\Z$-module, then $G$ is free abelian. 

For a set $S$, let $\mathcal{F}(S)$ be the free abelian group with basis $S$. 
Recall that Dedekind showed any subgroup of a free abelian group is free abelian. 

Let $U$ be a UFD, let $F={\rm frac}(U)$, and let $P$ be a complete set of representative prime elements in $U$. Then $F^*=U^*\times \mathcal{F}(P)$. 

%%%%%%%%%%%%%%%%%%%%%%%%%%%%%%%%%%%%%%%%%%%%%%%%%%%%%%%%%%%%%%%%%%%%%%%%%%%%%%

\section{The Polar Group of $A$}\label{polar-group}

Let $L^*$ be the unit group of $L={\rm frac}(B)$ and $K^*$ the unit group of $K={\rm frac}(A)$. 

\begin{definition} The {\bf polar group} of $A$ is the quotient group $L^*/B^*K^*$, denoted $\Pi (A)$. \end{definition}
An equivalent definition is as follows. 
Let $M\subset B'$ be the multiplicative submonoid $M=B^*\cdot A'$, noting that $\overline{M}=M$, and that $M\cap\Delta (B)=B^*$. 
Define a relation on $B'$ as follows.
\smallskip
\begin{center}
$f\sim g$ if and only if $f\bar{g}\in M$
\end{center}
\smallskip
This is an equivalence relation and the equivalence classes $[f]$ form an abelian group $G$, where $[f]^{-1}=[\bar{f}]$. Since the map
$L^*\to \Pi (A)$ defined by $f/g\to [f\bar{g}]$ is surjective and has kernel equal to $B^*K^*$, we see that $G=\Pi (A)$. 

Observe that, if $A$ is a field, then $B=A[i]$ is also a field and $\Pi (A)=\{ 1\}$. 

We use the following notation.
\begin{enumerate}
\item For any subset $T\subset \Pi (A)$, $\langle T\rangle\subset\Pi (A)$ denotes the subgroup generated by $T$. 
\item For any subset $S\subset B'$, $[S]\subset \Pi (A)$ denotes the set $\{ [x]\,\vert\, x\in S\}$,
and $\langle S\rangle$ denotes the subgroup $\langle [S]\rangle$.
\item If $A$ is affine over $\R$ and $X={\rm Spec}(A)$, then $\Pi (X)=\Pi (A)$. 
\end{enumerate}

We make an analogous definition for projective varieties. Suppose that $A$ is affine over $\R$ and is $\Z$-graded: $A=\bigoplus_{n\in\Z}A_n$.
Extend the grading to $B=\bigoplus_{n\in\Z}B_n$, noting that $B_0$ is a subring of $B$.
Then $X={\rm Proj}(A)$ is a real form of $Y={\rm Proj}(B)$. Define:
\[
L_0=\{ f/g\, |\, f,g\in B_n \,\, {\rm for\,\, some}\,\, n\in\Z \, ,\, g\ne 0\}  \quad {\rm and}\quad K_0=K\cap L_0
\]
Then $L_0$ is a subfield of $L$ which is algebraically closed in $L$, and $K_0$ is a subfield of $K$ which is algebraically closed in $K$. 
Define the {\bf polar group} of $X$ by:
\[
\Pi (X)=L_0^*/B_0^*K_0^*
\]
We have:
\[
L_0^*/B_0^*K_0^* = L_0^*/(L_0^*\cap B^*K^*) \cong L_0^*B^*K^*/B^*K^* \subset L^*/B^*K^*
\]
Therefore, $\Pi ({\rm Proj}(A))$ can be viewed as a subgroup of $\Pi ({\rm Spec}(A))$. 

\begin{example}\label{example2} {\rm Consider the rings:
\[
 A=\R [x]=\R^{[1]} \quad {\rm and}\quad B=\C [x]=\C^{[1]}
 \]
Since $B$ is a UFD, we see that $L^*=\C (x)^*=\C^*\times G$, where $G\cong\mathcal{F}(\C )$. In particular, 
$z_1^{e_1}\cdots z_n^{e_n}\in\mathcal{F}(\C )$ corresponds to $(x-z_1)^{e_1}\cdots (x-z_n)^{e_n}\in G$, where $e_j\in\Z$. 
Define subgroups $H_1,H_2\subset G$ by:
\[
H_1=\langle x-z\, |\, z\in\HH \rangle \quad {\rm and}\quad H_2=\langle (x-z)(x-\bar{z}) , x-r\, |\, z\in\C\, ,\, r\in\R\rangle
\]
Then $G=H_1\times H_2$ and $K^*=\R^*H_2$. Therefore, $\Pi (\A_{\R}^1)=H_1\cong\mathcal{F}(\HH )$, 
the free abelian group consisting of monic rational functions on $\C$ all of whose zeros and poles lie in $\HH$. 
}
\end{example}

\begin{example}\label{proj-line2} {\rm Continuing the preceding example, note that 
the subgroup of $\C (x)^*$ of functions of degree zero are precisely those which extend to the Riemann sphere $\C\cup\{ \infty\}=\PP^1_{\C}$. In this way, we see that 
$\Pi (\PP^1_{\R})$ is the subgroup of $\Pi (\A^1_{\R})$ 
consisting of products of the form $[x-z_1]^{e_1}\cdots [x-z_n]^{e_n}$, where $z_j\in\HH$ and $e_1+\cdots +e_n=0$. 
By Dedekind's theorem, $\Pi (\PP^1_{\R})$ is a free abelian group. For any fixed $\eta\in\HH$, a basis for $\Pi (\PP^1_{\R})$  is given by:
\[
\{ [x-z][x-\eta ]^{-1}\, |\, z\in\HH\setminus \{\eta\} \}
\]
Therefore, $\Pi (\PP^1_{\R})\cong\mathcal{F}(\HH\setminus \{ i\})$. }
\end{example} 

\begin{remark} {\rm $B$ can have distinct real forms, and the definition of $[f]$ depends on the particular real form $A$ of $B$. Thus, when more than one real form of $B$ is being considered, say $A$ and $\tilde{A}$, we will distinguish their equivalence classes by writing $[f]_A$ and $[f]_{\tilde{A}}$. Similarly, if $A$ is a real form of $B$ and $\tilde{A}$ is a real form of $\tilde{B}$, we also write  $[f]_A$ and $[f]_{\tilde{A}}$. Similarly, we may write $[X]_A$ for $[X]$ or $\langle X\rangle_A$ for $\langle X\rangle$ in certain situations. }
\end{remark}

\subsection{Localizations}
Let $S\subset A'$ be a multiplicatively closed set. Then $S^{-1}A$ is an $\R$-domain which is a real form of $S^{-1}B=\C\otimes_{\R}S^{-1}A$. Define the multiplicatively closed set 
$\hat{S}\subset B$ by:
\[
\hat{S}=\{ f\in B\,\vert\, s\in fB \,{\rm for\,\, some}\, s\in S\}
\]
Then $S^{-1}B=\hat{S}^{-1}B$. 

\begin{proposition}\label{localize} Under the foregoing hypotheses:
\begin{itemize} 
\item [{\bf (a)}] $\Pi (S^{-1}A)\cong \Pi (A)/\langle\hat{S}\rangle_A$
\smallskip
\item [{\bf (b)}] If $B$ is a UFD, then 
$\Pi (A)\cong\langle\hat{S}\rangle_A \times \Pi (S^{-1}A)$
\end{itemize}
\end{proposition}

\begin{proof} Consider the surjective homomorphism $\pi : \Pi (A)\to\Pi (S^{-1}A)$ sending $[f]_A$ to $[f]_{S^{-1}A}$. Since the kernel of $\pi$ is $\langle\hat{S}\rangle_A$, part (a) follows.

Suppose that $B$ is a UFD, and consider the short exact sequence:
\[
1\to \langle\hat{S}\rangle_A\hookrightarrow\Pi (A)\xrightarrow{\pi} \Pi (S^{-1}A)\to 1
\]
It must be shown that $\pi$ has a section. 

Let nonzero $f\in S^{-1}B$ be given, and write $f=b/s$ for $b\in B'$ and $s\in S$. 
Since $B$ is a UFD, there exists $t\in \hat{S}\cup\{ 1\}$ and $g\in B'$ such that $b=tg$ and $\gcd (g,\hat{S})=1$. Note that $g$ is uniquely determined up to associates, and that 
$[f]_{S^{-1}A}=[g]_{S^{-1}A}$. Define the function:
\[
\sigma : \Pi (S^{-1}A)\to\Pi (A) \,\, ,\,\, [f]_{S^{-1}A}\mapsto [g]_A
\]
Since $g$ is unique up to associates, $\sigma$ is well-defined. In addition, it is easily checked that (1) $\sigma$ a group homomorphism; (2) 
$\pi\sigma$ is the identity on $\Pi (S^{-1}A)$; and (3) $\sigma\pi$ is the identity on $\sigma (\Pi (S^{-1}A))$. This confirms part (b).
 \end{proof}

\subsection{Subalgebras}

\begin{proposition}\label{subalgebra}  Let $R\subset A$ be a subalgebra, and define:
\[
\varphi :\Pi (R)\to \Pi (A)\,\, ,\,\, \varphi ([f]_R) = [f]_A 
\]
\begin{itemize} 
\item [{\bf (a)}] $\varphi$ is a well-defined group homomorphism. 
\item [{\bf (b)}] The kernel of $\varphi$ is $\langle M\cap T\rangle_R$ , where $T=\C\otimes_{\R}R\subset B$. 
\item [{\bf (c)}] If $T$ is factorially closed in $B$, then $\Pi (R)$ is a subgroup of $\Pi (A)$.
\end{itemize}
\end{proposition}

\begin{proof}  For part (a):
\[
[f]_R=[g]_R \implies f\bar{g}\in R'T^*\subset M \implies [f]_A=[g]_A
\]
This shows that $\varphi$ is well-defined, and it is clearly a homomorphism. Part (b) follows by definition. 

Let $R'=R\setminus\{ 0\}$.  
By part (b), $\varphi$ is injective if and only if $R'T^* =M\cap T$, and this equality holds if $T$ is factorially closed in $B$. 
\end{proof}

\subsection{Prime Ideals}

Let $\mathfrak{p}\subset A$ be a prime ideal such that $\mathfrak{p}B$ is prime in $B$. 
Then $S=A\setminus\mathfrak{p}$ and $\hat{S}=B\setminus\mathfrak{p}B$ are multiplicatively closed sets. By {\it Prop.\,\ref{localize}}, the sequence
\[
1\to \langle B\setminus\mathfrak{p}B\rangle_A \to \Pi (A) \to \Pi (A_{\mathfrak{p}})\to 1
\]
is exact.

\begin{proposition}\label{prime} Let $\pi : B\to B/\mathfrak{p}B$ be the standard surjection, and 
define the function:
\[ 
\psi : \langle B\setminus\mathfrak{p}B\rangle_A\to \Pi (A/\mathfrak{p}) \,\, ,\,\, \psi ([f]_A)=[\pi (f)]_{A/\mathfrak{p}}
\]
Then $\psi$ is a well-defined surjective group homomorphism, and the kernel of $\psi$ is $\langle (M+\mathfrak{p}B)\cap B'\rangle_A$.
\end{proposition}

\begin{proof}
$\psi$ is a well-defined function, since $f\in\hat{S}$ iff $\pi (f)\in (B/\mathfrak{p}B)'$, and $\hat{S}\subset B'$. 
In addition, it is clear from the definition that $\psi$ is a surjective group homomorphism, and that the kernel of $\psi$ is $\langle (M+\mathfrak{p}B)\cap B'\rangle_A$.
\end{proof}

%%%%%%%%%%%%%%%%%

\subsection{Transversal Generating Sets} 

Recall that the action of $\Gamma =B^*\rtimes\Z_2$ restricts to $\Delta (B)_1$, and that $B'$ is generated as a monoid by $M=B^*A'$ and $\Delta (B)_1$. 
Therefore, $\langle\Delta (B)_1\rangle =\Pi (A)$. 
Let $\mathcal{T}\subset\Delta (B)_1$ be a transversal of the $\Gamma$-action.\footnote{Here, we invoke the Axiom of Choice.} 
Given $f\in\Delta (B)_1$, let $\mathcal{O}_f$ denote the orbit of $f$. Given $g\in\Delta (B)_1$, let $h\in\mathcal{T}$ be such that $g\in\mathcal{O}_h$, 
and let $\sigma =(b,\lambda )\in\Gamma$ be such that $g=\sigma\cdot h$. Then
$g=bh$ or $g=b\bar{h}$, which implies $[g]=[h]$ or $[g]=[h]^{-1}$. We thus conclude that $\langle\mathcal {T}\rangle =\langle\Delta (B)_1\rangle =\Pi (A)$. 
The set $[ \mathcal{T}]$ is a {\bf transversal generating set} of $\Pi (A)$. 

%%%%%%%%%%%%%%%%%%%%%%%%%%%%%%%%%%%%%%%%%%%%%%%%%%%%%%%%%%%%%%%%%%%%%%%%%%%%%%

\section{Unique Factorization Domains}\label{UFD}

\subsection{Freeness Criterion} 
\begin{theorem}\label{free} Suppose that $B$ is a UFD. The following conditions are equivalent.
\begin{enumerate}
\item Every element of $B'$ has a unique polar factorization up to units.
\item Given $f\in B'$, $f\in \Delta (B)$ if and only if $\gcd (f,\bar{f})=1$. 
\item Given $f,g\in\Delta (B)$, $[f]=[g]$ if and only if $fB=gB$.
\item $\Pi (A)$ is square-torsion free.
\item $\Pi (A)$ is torsion free.
\item $\Pi (A)$ is a free abelian group.
\end{enumerate}
When these conditions hold, a basis for $\Pi (A)$ is given by a transversal generating set.
\end{theorem}

\begin{proof} The implications (6) implies (5) implies (4) are valid by definition.

Assume that condition (4) holds, and let $f,g\in\Delta (B)$ be such that $[f]=[g]$. Write $f\bar{g}=a\omega$ for $a\in A'$ and $\omega\in B^*$. Suppose that $p\in B$ is prime and 
$a\in pB$. If $pB=\bar{p}B$, then $[p]=[\bar{p}]=[p]^{-1}$ implies $[p]=1$. 
So $pB=rB$ for some $r\in A'$. But then $f\bar{g}\in rB$ implies $f\in rB$ or $g\in rB$, contradicting the assumption that $f,g\in\Delta (B)$. Therefore, 
$pB\ne\bar{p}B$. It follows that $a\in pB\cap\bar{p}B=p\bar{p}B$. Since $f,\bar{g}\not\in p\bar{p}B$, we may assume that $f\in pB$ and $\bar{g}\in \bar{p}B$, which implies that $f,g\in pB$. 

Write $a=p\bar{p}\alpha$ for $\alpha\in B'$, and write $f=pF$ and $g=pG$ for $F,G\in B'$. Then $\alpha\in K\cap B'=A'$. In addition, $f,g\in\Delta (B)$ implies $F,G\in\Delta (B)$. We have:
\[
f\bar{g}=(pF)(\bar{p}\bar{G})=p\bar{p}\alpha\omega \implies F\bar{G}=\alpha\omega \implies [F]=[G]
\]
By induction on the maximum number of prime factors of $f$ or $g$, we conclude that $FB=GB$. Therefore, $fB=gB$. This proves (4) implies (3). 

Assume that condition (3) holds, and suppose that $\gamma^n=1$ for some $\gamma\in\Pi (A)$ and $n\ge 1$. Let $f\in B$ be such that $[f]=\gamma$. Since $B$ is a UFD, $f$ admits a polar factorization $f=r h$, where $r\in A'$ and $h\in\Delta (B)$. Therefore, $[f]=[h]$ and $[h^n]=[f^n]=[1]$. By condition (3):
\[
h^nB=B \implies h\in B^*  \implies \gamma =[f]=[rh]=1
\]
This shows that (3) implies (5). Therefore, conditions (3), (4) and (5) are equivalent. 

Given $f\in\Delta (B)$, suppose that $\gcd (f,\bar{f})\ne 1$. Let $p\in B$ be a prime such that $f,\bar{f}\in pB$. Then $p\in\Delta (B)$. 
Since $M\cap\Delta (B)=B^*$, $p\not\in M$ and $[p]\ne 1$. Condition (3) implies $[p]^2\ne 1$, meaning that $pB\ne\bar{p}B$. 
But then $f\in pB\cap \bar{p}B=p\bar{p}B$, contradicting the fact that 
$f\in\Delta (B)$. We conclude that $\gcd (f,\bar{f})=1$. This shows that (3) implies (2). 

Assume that condition (2) holds. Suppose that $\gamma^2=1$ for $\gamma\in\Pi (A)$, and write $\gamma=[f]$ for $f\in\Delta (B)$. Then $f^2\in M$, so there exist $a\in A'$ and $\omega\in B^*$ with $f^2=a\omega$. If $a\not\in A^*$, then there exists a prime $p\in B$ with $a\in pB$. Since $\bar{f}^2=a\bar {\omega}$, we see that $f^2,\bar{f}^2\in pB$, which implies 
$f,\bar{f}\in pB$. However, (2) implies $\gcd (f,\bar{f})=1$, since $f\in \Delta (B)$. Therefore, $a\in A^*$, and $f^2\in B^*$ gives $f\in B^*$ and $[f]=1$. This shows that (2) implies (4). Therefore, conditions (2)-(5) are equivalent.

Assume that conditions (2)-(5) hold. Since $B$ is a UFD, it satisfies the ACCP. By {\it Lemma\,\ref{ACCP1}}, every element of $B'$ admits a polar factorization. 
Let $\alpha f=\beta g$ for $\alpha ,\beta\in A$ and $f,g\in\Delta (B)$. Then $[f]=[g]$, and condition (3) implies $fB=gB$. Write $g=\omega f$ for $\omega\in B^*$. Then $\alpha f=\beta\omega f$ , so $\omega\in K\cap B^*=A^*$ and $\alpha A=\beta B$. Therefore, conditions (2)-(5) imply (1). 

Conversely, assume that condition (1) holds, and that $[f]=[g]$ for $f,g\in\Delta (B)$. Then $f\bar{g}=\sigma a$ for $\sigma\in B^*$ and $a\in A'$, which implies 
$(g\bar{g})f=(a\sigma )g$. Since these are polar factorizations, condition (1) implies $fB=gB$, and condition (3) holds. Therefore, conditions (1)-(5) are equivalent.

Assume that conditions (1)-(5) hold. Let $\mathcal{T}\subset\Delta (B)_1$ be a transversal of the $\Gamma$-action. Suppose that $f_1,...,f_n\in\mathcal{T}$ are distinct ($n\ge 0$), and let 
$e_1,...,e_n\in\N$ be nonzero. Then:
\begin{equation*}\label{primitive}
 (\dag ) \hspace{.5in} f:=f_1^{e_1}\cdots f_n^{e_n}\in\Delta (B) 
\end{equation*}
In order to see this, let $r\in A'$ be such that $f\in rB$. If $r$ is a non-unit, let $p\in B$ be a prime with $r\in pB$. 
Since each $f_j$ is irreducible, we see that $pB=f_jB$ for some $j$. Therefore, condition (3) implies $[p]^2=[f_j]^2\ne 1$. 
However, $r\in\bar{p}B$ as well, so $r\in pB\cap\bar{p}B$. If $pB=\bar{p}B$, then $[p]^2=1$ (as above), which is not the case. Therefore, $\bar{p}B=f_rB$ for some $r\ne j$. But then 
$[f_r]=[f_j]^{-1}$, which means that $f_r$ is in the orbit of $f_j$, a contradiction. Therefore, $r\in A^*$, and $f\in \Delta (B)$, as claimed. 

Now suppose that $[f_1]^{e_1}\cdots [f_s]^{e_s}[f_{s+1}]^{-e_{s+1}}\cdots [f_n]^{-e_n}=1$ for distinct elements $f_j\in\mathcal{T}$, $n\ge 0$, and nonzero $e_j\in\N$. 
Then:
\[
[f_1]^{e_1}\cdots [f_s]^{e_s}[\bar{f}_{s+1}]^{e_{s+1}}\cdots [\bar{f}_n]^{e_n}=1 
\]
If $g=f_1^{e_1}\cdots f_s^{e_s}\bar{f}_{s+1}^{e_{s+1}}\cdots \bar{f}_n^{e_n}$, then $[g]=1$, and property ($\dag$) implies $g\in\Delta (B)$. By condition (3), it follows that $gB=B$. 
Therefore, $g\in B^*$ implies $f_j\in B^*$ for each $j$. This is only possible if $n=0$, since $B^*\cap\mathcal{T}=\emptyset$. 

Therefore, there are no relations among elements of $[\mathcal{T}]$. In addition, it was shown in {\it Sect.\,2} that $[\mathcal{T}]$ generates $\Pi (A)$. 
Therefore, $\Pi (A)$ is a free abelian group with basis 
$[\mathcal{T}]$. 

This shows that conditions (1)-(5) imply (6), and completes the proof of the theorem. 
\end{proof} 

%%%%%%%%%

\subsection{The Case $A$ is a UFD}

\begin{lemma}\label{delta} Assume that $A$ is a UFD and let $f\in B'$ be given. Write $f=f_1+if_2$ for $f_1,f_2\in A$. Then $f\in\Delta (B)$ if and only if $\gcd (f_1,f_2)=1$.
\end{lemma}

\begin{proof} 
If $f_1,f_2\in rA$ for $r\in A'$, $r\not\in A^*$, then $f\in rB$ implies $f\not\in\Delta (B)$. 
Therefore, $f\in\Delta (B)$ implies $\gcd (f_1,f_2)=1$. 

Conversely, assume that $\gcd (f_1,f_2)=1$, and let $f\in rB$ for $r\in A'$. If $f=rg$ for $g\in B'$, then $f_1,f_2\in rA$ implies $rA=A$. 
Therefore, $f\in \Delta (B)$. 
\end{proof}

\begin{lemma}\label{A=UFD1} Assume that $A$ is a UFD. If $f\in\Delta (B)$, then $f^n\in\Delta (B)$ for each integer $n\ge 0$. 
\end{lemma}

\begin{proof} Let $F=N_f$ be the matrix form of $f=f_1+if_2$, and suppose that $F^n=pG$ for some integer $n\ge 1$, where $p\in A$ is prime and $G=N_g$ for $g\in B'$. 
Let $\pi :A\to A/pA$ be the standard surjection. 
Then $\pi (F)$ is a nonzero nilpotent matrix over the integral domain $A/pA$. Therefore:
\[
\det (\pi (F))= {\rm tr}(\pi (F))=0 \implies \det (F) , {\rm tr}(F)\in pA \implies f_1^2+f_2^2 , 2f_1\in pA \implies f_1,f_2\in pA
\]
But this is a contradiction, since {\it Lemma\,\ref{delta}} implies that $f_1$ and $f_2$ are relatively prime. Therefore, no such integer $n$ exists. It follows that, if $f^n=(f^n)_1+i(f^n)_2$ for $(f^n)_1,(f^n)_2\in A$, then  
$\gcd ((f^n)_1,(f^n)_2)=1$. {\it Lemma\,\ref{delta}} implies $f^n\in\Delta (B)$ for all $n\ge 1$.
\end{proof}

\begin{lemma}\label{UFD4} Assume that $A$ is a UFD. Given $f,g\in\Delta (B)$, let $fg=rh$ be a polar factorization. 
\begin{itemize}
\item [{\bf (a)}] $f\bar{f},g\bar{g}\in pA$ for every prime $p\in A$ dividing $r$. 
\item [{\bf (b)}] If $\gcd (f\bar{f},g\bar{g})=1$, then $fg\in\Delta (B)$. 
\end{itemize}
\end{lemma}

\begin{proof} Suppose that $p\in A$ is a prime dividing $r$. Let the matrix forms of $f,g,r,h$ be given by $F,G,R,H$, respectively, and let $S\in M_2(A)$ be such that $R=pS$. 
Then $FG=pSH$ implies $\det (F)\det (G)=p^2\det (SH)$. Since $g\in\Delta (B)$, {\it Lemma\,\ref{delta}} implies $G\ne pT$ for any $T\in M_2(A)$. Therefore, $\det (F)\in pA$, and by symmetry, $\det (G)\in pA$. This proves part (a), and part (b) is directly implied by part (a). 
\end{proof}

\begin{theorem}\label{A=UFD3} If $A$ is a UFD, then $\Pi (A)$ is torsion free.
\end{theorem}

\begin{proof}
Let $f\in B'$ be such that $[f]^n=1$ for $n\ge 1$. By {\it Lemma\,\ref{ACCP1}}, $f$ admits a polar factorization: $f=rh$, $r\in A'$, $h\in \Delta (B)$. We have $[h^n]=1$, which implies $h^n\in M$. By {\it Lemma\,\ref{A=UFD1}}, we also have $h^n\in\Delta (B)$. Therefore:
\[
h^n\in M\cap\Delta (B)=B^* \implies h\in B^* \implies [f]=[h]=1
\]
\end{proof}

Combining {\it Thm.\,\ref{A=UFD3}} with {\it Thm.\,\ref{free}} gives the following. 

\begin{corollary}\label{A=UFD2} If $A$ is a UFD and $B$ is a UFD, then $\Pi (A)$ is a free abelian group, and a basis for $\Pi (A)$ is a transversal generating set. 
\end{corollary}

In particular, this corollary implies that, if $B$ is a UFD and $B^*=\C^*$, then $\Pi (A)$ is a free abelian group. 

\begin{corollary}\label{circle} $\Sp^1$ is not factorial.
\end{corollary}

\begin{proof} Let $A=\R [x,y]$ with $x^2+y^2=1$, and let $B=\C [t,t^{-1}]$, where $t=x+iy$. Then $\bar{t}=t^{-1}$. Given $z\in\C$ with $\| z\| =1$, we have:
\[
[t-z]^{-1}=[\bar{t}-\bar{z}]=[t^{-1}-z^{-1}]=[-z^{-1}t^{-1}(t-z)]=[t-z]
\]
Therefore, $[t-z]^2=1$ and $\Pi (A)$ has torsion. By {\it Thm.\,\ref{A=UFD3}}, $\Sp^1$ is not factorial.
\end{proof}

\begin{corollary} There is no dominant morphism $\R^*\to\Sp^1$. Equivalently, $\Sp^1$ cannot be parametrized by Laurent polynomials $\R [t,t^{-1}]$. 
\end{corollary}

\begin{proof} Let $A=\R [t,t^{-1}]$. If there exists a dominant morphism $\R^*\to\Sp^1$, then there is a subalgebra $R\subset A$ with $R\cong \R [\Sp^1]$. 
If $T=\C\otimes_{\R}R$ and $B=\C\otimes_{\R}A$, then $B=\C [t,t^{-1}]$ and $T=\C [s,s^{-1}]$ for some $s\in B$. We my assume that $s=t^m$ for some positive $m\in\Z$. 

Consider the group homomorphism $\varphi :\Pi (R)\to \Pi (A)$. Since $A$ is a UFD, $\Pi (A)$ is torsion free. The proof of {\it Cor.\,\ref{circle}} shows $[s-\zeta]_R$ has order 2 for each $\zeta\in\C$ with $\|\zeta\|=1$. It follows that: 
\[
\varphi \left( [s-\zeta ]_R\right)=[t^m-\zeta]_A=1 \quad \forall \zeta\in\C\,\, ,\,\, \|\zeta\|=1
\]
Choose $\zeta\not\in \{ \pm 1\}$. Then $[t^m-\zeta ]_A=1$ implies $\omega (t^m-\zeta)\in A'$ for some $\omega\in\C^*$. But then
\[
\omega (t^m-\zeta )\in A\cap\C [t]=\R [t] \implies \omega\in\R^* \implies\zeta\in\R^*
\]
a contradiction. 
\end{proof}

%%%%%%%%%%%%%%%%%%%%%%%%%%%%%%%%%%%%%%%%%%%%%%%%%%%%%%%%%%%%%%%%%%%%%%%%%%%%%%

\section{Computations}

\subsection{Two Real Forms of $\PP^1_{\R}$}\label{projective-line}
The two real forms of $\PP^1_{\C}$ are $\PP^1_{\R}$ and $X_{\R}$, which is the conic in $\PP^2_{\R}$ defined by $x_0^2+x_1^2+x_2^2=0$;
see {\it Example\,\ref{proj-line}}. It was shown in {\it Example\,\ref{proj-line2}} that $\Pi (\PP^1_{\R})\cong\mathcal{F}(\HH\setminus \{ i\})$. 

The homogeneous coordinate ring of $X_{\R}$ is $A=\R [x_0,x_1,x_2]/(x_0^2+x_1^2+x_2^2)$, and if $K={\rm frac}(A)$, then:
\[
K_0=\R \left(\frac{x_1}{x_0},\frac{x_2}{x_0}\right) \quad{\rm where}\quad \left(\frac{x_1}{x_0}\right)^2+\left(\frac{x_2}{x_0}\right)^2+1=0
\]
If $B=\C [x_0,x_1,x_2]/(x_0^2+x_1^2+x_2^2)$, then $B_0=\C$ and:
\[ 
L=\C (x_1+ix_2,x_0)\cong\C^{(2)} \quad {\rm and}\quad L_0=\C ((x_1+ix_2)/x_0)\cong\C^{(1)}
\]
Let $t=(x_1+ix_2)/x_0$. Then $t\bar{t}=-1$, which implies $\bar{t}=-t^{-1}$. From {\it Example\,\ref{example2}}, recall that $L_0=\C^*\times G$ where $G\cong\mathcal{F}(\C)$. 
Given $z\in \C^*$, let $\kappa =(t-z)(\bar{t}-\bar{z})\in K_0$. We have:
\[
\kappa =(t-z)(\bar{t}-\bar{z})=(t-z)(-t^{-1}-\bar{z})=-\bar{z}t^{-1}(t-z)(t+\bar{z}^{-1}) \implies (t-z)(t+\bar{z}^{-1}) = -\bar{z}^{-1}\kappa t
\]
Let $P: L_0^*\to\Pi (X_{\R})$ be the natural surjection, and let $H\subset L_0^*$ be the subgroup $H=\langle t-z\, |\, z\in\D\rangle$. 
Since $\| z\|\ge 1$ if and only if $\| -\bar{z}^{-1}\|\le 1$, it follows that $P(H)=\Pi (X_{\R})$. 
In addition, $z=-\bar{z}^{-1}$ if and only if $z\in\partial\D$, i.e., $\| z\|=1$. 
Let $U\subset\D$ be the interior of $\D$, and let $Z\subset\partial\D$ be the arc $Z=\{ e^{i\theta}\, |\, 0\le \theta <\pi\}$. 
Then $P$ is bijective on the subgroup $\langle t-z\, |\, z\in U\cup Z\rangle$, which by Dedekind's theorem is a free abelian group. 
By identifying antipodal points of $\partial\D$, we conclude that $\Pi (X_{\R})$ is the free abelian group generated by the real projective plane:
$\Pi (X_{\R})\cong\mathcal{F}(\PP^2_{\R})$. See {\it Fig.\,1}. 

\begin{figure}[b]
\begin{subfigure}[b]{0.4\textwidth}
\centering
\begin{tikzpicture}
\begin{scope}
\clip (-2,-2) rectangle (2,2);
\draw[fill=lightgray,lightgray] (-1.5,-1.0) rectangle (1.5,1);
\draw[fill=white] (0,0) circle (.05cm);
\draw[dashed, thick, -] (-2,-1.0)--(2,-1.0);
\node at (.3,0){$i$};
\end{scope}
\node at (-2.4,1){$\HH\setminus\{ i\}$};
\end{tikzpicture}    

\caption{$\PP^1_{\R}$}
\end{subfigure}
\begin{subfigure}[b]{0.3\textwidth}
\centering
\begin{tikzpicture}
\begin{scope}
\clip (-2,-2) rectangle (2,2);
\draw[fill=lightgray] (0,0) circle (1cm);
\draw[thick, ->]     (0:1cm) arc (0:  90:1cm);
\draw[thick, ->]     (0:1cm) arc (0:  270:1cm);
\draw[thick, -]  (0:1cm) arc (0: -90:1cm);
\filldraw (1,0) circle (1.0pt);
\filldraw (-1,0) circle (1.0pt);
\node at (1.3,0){$1$};
\node at (-1.4,0){$-1$};
\node at (-1.2,1.2){$\PP^2_{\R}$};
\end{scope}
\end{tikzpicture}    

\caption{$X_{\R}: x_0^2+x_1^2+x_2^2=0$}
\end{subfigure}
\caption{Group generators for two real forms of $\PP^1_{\C}$}
\end{figure}

\subsection{Localizations of $\R [x]$}

\begin{example}\label{example3} {\rm $A=S^{-1}\R [x]$, where $S=\{ (x-a_1)^{e_1}\cdots (x-a_n)^{e_n}\, |\, e_j\in\N\}$ for $a_1,...,a_n\in\R$. 
Recall from {\it Example\,\ref{example2}} that $\Pi (\R[x])\cong\mathcal{F}(\HH )$. Since $\hat{S}=S$ and $[S]=\{ 1\}$, {\it Prop.\,\ref{localize}} implies:
\[
\Pi (A)=\Pi (S^{-1}\R [x]) = \Pi ( \R [x])\cong \mathcal{F}(\HH )
\]
For example, if $S=\{ (x-1)^{e_1}(x+1)^{e_2}\, |\, e_1,e_2\in\N\}$, then ${\rm Spec}(A)=\R^{**}$, the twice-punctured real line. See See {\it Fig.\,2(A)}.}
\end{example}

\begin{example} {\rm $A=\R [x,(x^2+1)^{-1}]$. Then $B=\C [x,(x-i)^{-1},(x+i)^{-1}]$, which is the coordinate ring of $\C^{**}$, the twice-punctured complex line. 
If $S=\{ (x^2+1)^m\,\vert\, m\in\N\}$, then $A=S^{-1}\R [x]$ and:
\[
\hat{S}=\{ (x-i)^m, (x+i)^n\,\vert\, m,n\in\N\}
\]
By {\it Prop.\,\ref{localize}}, it follows that 
$\Pi (A)=\mathcal{F}(\HH\setminus\{ i\})$. See {\it Fig.\,2(B)}.}
\end{example}

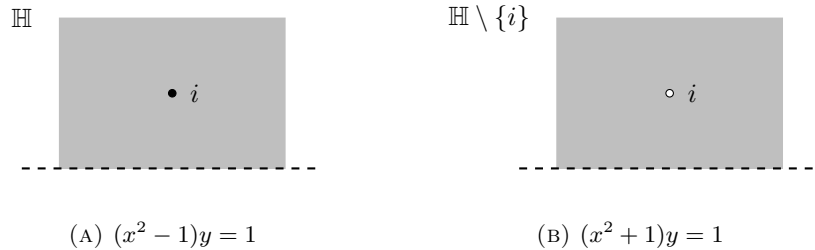
\begin{figure}[b]
\begin{subfigure}[b]{0.4\textwidth}
\centering
\begin{tikzpicture}
\begin{scope}
\clip (-2,-2) rectangle (2,2);
\draw[fill=lightgray,lightgray] (-1.5,-1.5) rectangle (1.5,.5);
\draw[fill=black] (0,-.5) circle (.05cm);
\draw[dashed, thick, -] (-2,-1.5)--(2,-1.5);
\node at (.3,-.5){$i$};
\end{scope}
\node at (-2,.5){$\HH$};
\end{tikzpicture}    

\caption{$(x^2-1)y=1$}
\end{subfigure}
\begin{subfigure}[b]{0.4\textwidth}
\centering
\begin{tikzpicture}
\begin{scope}
\clip (-2,-2) rectangle (2,2);
\draw[fill=lightgray,lightgray] (-1.5,-1.5) rectangle (1.5,.5);
\draw[fill=white] (0,-.5) circle (.05cm);
\draw[dashed, thick, -] (-2,-1.5)--(2,-1.5);
\node at (.3,-.5){$i$};
\end{scope}
\node at (-2.4,.5){$\HH\setminus\{ i\}$};
\end{tikzpicture}    

\caption{$(x^2+1)y=1$}
\end{subfigure}
\caption{Group generators for two real forms of $\C^{**}$}
\end{figure}

\begin{example} {\rm $A=\R [x]_{\mathfrak{p}}$, where $\mathfrak{p}=(x^2+1)\R [x]$. Then $B=\hat{S}^{-1}\C[x]$, where $S$ is the set of products of primes of $B$ not associate with $x+i$ or $x-i$. 
Since $[x+i]^{-1}=[x-i]$, it follows that, given $f\in B'$, $[f]=[x-i]^e$ for some $e\in\Z$. Therefore, $\Pi (A)\cong\Z$, generated by $[x-i]$. }
\end{example}

\subsection{Three Real Forms of $\C^*$}

\begin{example} {\rm $A=\R [x,x^{-1}]$. By {\it Example\,\ref{example3}} we have
\[
\Pi (A)=\Pi (S^{-1}\R [x]) = \Pi ( \R [x])=\mathcal{F}(\HH )
\]
where $S=\{ x^n\, |\, n\in\N\}$. See {\it Fig.\,2(A)}}.
\end{example}

\begin{example}{\rm $A=\R [x,y]$ where $x^2+y^2+1=0$. By {\it Prop.\,\ref{UFD}} below, $A$ is a UFD. Therefore, by {\it Cor.\,\ref{A=UFD2}}, $\Pi (A)$ is a free abelian group. 
We have $L=\C (x+iy)$ and $K=\R (x,y)$. There is a natural surjection $\sigma :\Pi (X_{\R})\to \Pi (A)$, where $\sigma [t-z]=[x+iy-z]$ for $z\in\C$; see {\it Example\,\ref{projective-line}}. 
The kernel of $\sigma$ is $\langle t\rangle$. Therefore, $\Pi (A)$ is generated by the punctured real projective plane, which is homeomorphic to an open M\"obius band: $\Pi (A)\cong\mathcal{F}(\PP^2_{\R}\setminus\{ 0\})$.
See {\it Fig.\,3(B)}. 
}
\end{example}

\begin{example}\label{S1} {\rm $A=\R [x,y]$ where $x^2+y^2-1=0$. Write $B=\C[t,t^{-1}]$, where $t=x+iy$, noting that $t^{-1}=\bar{t}$. We have:
\[
\Delta (B)_1=\{ ct^n(t-z)\, |\, c,z\in\C^* , n\in\Z\}
\]
Given $z\in\C^*$, let $f=t-z$, and suppose that $g\in\mathcal{O}_f$ (the orbit of $f$), where $g=ct^n(t-w)$ for $c,w\in\C^*$ and $n\in\Z$. 
There exists $(b,\lambda )\in\Gamma$ with $g=(b,\lambda )\cdot f$. Write $b=\beta t^m$ for $\beta\in\C^*$ and $m\in\Z$. If $\lambda =1$, then $g=bf$ and we have:
\[
ct^n(t-w)=\beta t^m(t-z) \implies w=z 
\]
If $\lambda\ne 1$, then $g=b\bar{f}$, and we have:
\begin{equation}\label{case2}
ct^n(t-w)=\beta t^m(\bar{t}-\bar{z})=\beta t^m(t^{-1}-\bar{z})=-\beta t^{m-1}\bar{z}(t-\bar{z}^{-1}) \implies w=\bar{z}^{-1}
\end{equation}
Since $\bar{z}^{-1}=z$ if and only if $\| z\| =1$, it follows that the set $\mathcal{T}:=\{ t-z \, |\, z\in\D , z\ne 0\}$ is a transversal for the $\Gamma$-action.
However, $[\mathcal{T}]$ is not a minimal generating set, as will be shown. 

Let $V=\partial\D$ and $U=\{ z\in\C \, |\, 0<\| z\| <1\}$. Define subgroups $G,H\subset\Pi (A)$ by:
\[
G=\langle t-z \, |\, z\in V\rangle \quad {\rm and} \quad H=\langle t-z \, |\, z\in U\rangle
\]
Since $\D=U\cup V$, we see that $\Pi (A)=GH$.

Line (\ref{case2}) shows $[t-z]^{-1}=[t-\bar{z}^{-1}]$, and this implies 
$[t-\zeta]^2=1$ for $\zeta\in V$. Therefore, $\gamma^2=1$ for every $\gamma\in G$. 
In addition, note that $t^2+1=2xt$ and $t^2-1=2iyt$. 
Therefore, given $\zeta\in V$:
\[
(t-\zeta )(t-\bar{\zeta})=t^2-(\zeta +\bar{\zeta})t+1= 2(x-{\rm Re}(\zeta))t \in M \implies [t-\bar{\zeta}]=[t-\zeta]^{-1}=[t-\zeta]
\]
Similarly:
\[
(t+\zeta )(t-\bar{\zeta})=t^2+(\zeta -\bar{\zeta})t-1= 2i(y+{\rm Im}(\zeta ))t \in M \implies [t-\bar{\zeta}]=[t+\zeta ]^{-1}=[t+\zeta]
\]
Therefore, if $E$ is the arc $E=\{ t-e^{i\theta}\, |\, 0\le\theta <\pi/2\}$, then $G=\langle E\rangle$. 

On the other hand, $H\cong\mathcal{F}(U)$. In order to show this, assume that there is a relation among elements of $[U]$. 
Then there exists a Laurent polynomial $p(t)\in A$ such that $p(t)$ is not a unit and all the nonzero roots of $p(t)$ are in $U$. Let $z\in U$ be such that $p(t)\in (t-z)B$. As in line (1), we have:
\[
p(t)=\overline{p(t)}\in (\bar{t}-\bar{z})B=(t-\bar{z}^{-1})B
\]
Therefore, both $z$ and $\bar{z}^{-1}$ are in $U$, which is a contradiction. Consequently, no such $p(t)$ exists. 

Since $G$ is torsion and $H$ is free abelian, it follows that $G\cap H=\{ 1\}$, and that $\Pi (A)=G\times H$. 
$\Pi (A)$ is generated by the punctured disc $\D\setminus\{ 0\}$ with boundary identifications as indicated in {\it Fig.\,3(C)}. }
\end{example}

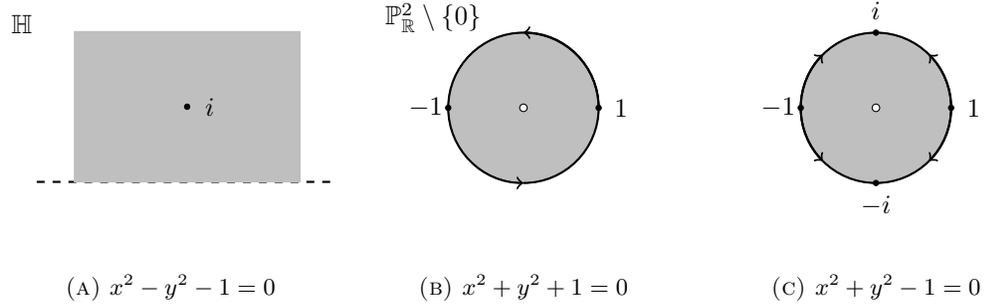
\begin{figure}[t]

\begin{subfigure}[b]{0.3\textwidth}
\centering
\begin{tikzpicture}
\draw[dashed, thick, -] (-2,1)--(2,1);
\draw[fill=lightgray,lightgray] (-1.5,1) rectangle (1.5,3);
\node at (-2.2,3.1){$\HH$};
\node at (0,.1){};
\filldraw (0,2) circle (1.0pt);
\node at (.3,2){$i$};
\end{tikzpicture}    

\caption{$x^2-y^2-1=0$}
\end{subfigure}
\begin{subfigure}[b]{0.3\textwidth}
\centering
\begin{tikzpicture}
\begin{scope}
\clip (-2,-2) rectangle (2,2);
\draw[fill=lightgray] (0,0) circle (1cm);
\draw[fill=white] (0,0) circle (.05cm);
\draw[thick, ->]     (0:1cm) arc (0:  90:1cm);
\draw[thick, ->]     (0:1cm) arc (0:  270:1cm);
\draw[thick, -]  (0:1cm) arc (0: -90:1cm);
\filldraw (1,0) circle (1.0pt);
\filldraw (-1,0) circle (1.0pt);
\node at (1.3,0){$1$};
\node at (-1.3,0){$-1$};
\node at (-1.2,1.2){$\PP^2_{\R}\setminus\{ 0\}$};
\end{scope}
\end{tikzpicture}    

\caption{$x^2+y^2+1=0$}
\end{subfigure}
\begin{subfigure}[b]{0.3\textwidth}
\centering
\begin{tikzpicture}
\begin{scope}
\clip (-2,-2) rectangle (2,2);
\draw[fill=lightgray] (0,0) circle (1cm);
\draw[fill=white] (0,0) circle (.05cm);
\draw[thick, ->]     (0:1cm) arc (0:  45:1cm);
\draw[thick, ->]     (0:1cm) arc (0:  -225:1cm);
\draw[thick, ->]     (0:1cm) arc (0:  225:1cm);
\draw[thick, ->]     (0:1cm) arc (0: -45:1cm);
\filldraw (1,0) circle (1.0pt);
\filldraw (-1,0) circle (1.0pt);
\filldraw (0,1) circle (1.0pt);
\filldraw (0,-1) circle (1.0pt);
\node at (0,1.3){$i$};
\node at (0,-1.3){$-i$};
\node at (1.3,0){$1$};
\node at (-1.3,0){$-1$};
%\node at (-1.2,1.2){$\mathcal{C}$};
\end{scope}
\end{tikzpicture}    
\caption{$x^2+y^2-1=0$}
\end{subfigure}
\caption{Group generators for three real forms of $\C^*$}
\end{figure}

\subsection{Polynomial Curves} Let $C\subset \R^n$ be a polynomial curve, that is:
\[
C={\rm Spec}(R) \quad {\rm for}\quad R=\R [p_1(x),...,p_n(x)]\subset A=\R [x]
\]
Note that $C$ is a rational curve by L\"uroth's Theorem.
If $T=\C\otimes_{\R}R$, then:
\[ 
T=\C[p_1(x),...,p_n(x)]\subset B=\C [x] \quad {\rm and}\quad M\cap T=(A'\C^*)\cap T=R'\C^*
\]
Therefore, $\langle M\cap T\rangle_R=\{ 1\}$. By {\it Prop.\,\ref{subalgebra}}, 
$\Pi (C)$ is a subgroup of $\Pi (\R [x])\cong\mathcal{F}(\HH )$. By Dedekind's theorem, $\Pi (C)$ is a free abelian group. 

For a given polynomial curve $C$, we seek a description of a basis of $\Pi (C)$ in terms of the basis $\HH$. 
For example, consider the cuspidal cubic curve $C$ defined by:
\[
R=\R [x^2,x^3]=\{ p(x)\in\R[x]\,\vert\, p^{\prime}(0)=0\}
\]
Then 
\[
T=\C [x^2,x^3]=\{ p(x)\in\C[x]\,\vert\, p^{\prime}(0)=0\} 
\]
and $\Delta(T)$ is the set of $f\in T$ with no conjugate roots, i.e., $f(z)=0$ implies $f(\bar{z})\ne 0$. Given monic $f\in\Delta (T)$, if 
\[
f=(x-w_1)\cdots (x-w_n) \,\, , \,\, w_j\in\C\setminus\R
\]
then the condition $f'(0)=0$ is equivalent to 
the condition $w_1^{-1}+\cdots +w_n^{-1}=0$. Assume that $w_1,...,w_s\in\HH$ and $w_{s+1}\cdots w_n\not\in\HH$. Since $[x-z]=[x-\bar{z}]^{-1}$ for any $z\in\C$, it follows that the subgroup $\Pi (C)$ of $\Pi (\R [x])$ is generated by elements of the form
\[
[x-z_1]\cdots [x-z_s][x-z_{s+1}]^{-1}\cdots [x-z_n]^{-1} \,\, ,\,\, z_1,...,z_n\in\HH
\]
where $z_1^{-1}+\cdots +z_s^{-1}+\bar{z}_{s+1}^{-1}+\cdots \bar{z}_n^{-1}=0$ and $\{ z_1,...,z_s\}\cap \{z_{s+1},...,z_n\}=\emptyset$. 
However, many such products correspond to reducible elements of $\Delta (T)$. 
For example, if $f\in\Delta (T)$ and $\deg f=4$, then $f\in\Delta (T)_1$ if and only if $w_1+w_1+w_3+w_4\ne 0$. 
Finding a basis for $\Pi (C)$ of this form thus requires a good description of irreducibles in $T$, which (to the author's knowledge) is lacking.

%%%%%%%%%%%%%%%%%%%%%%%%%%%%%%%%%%%%%%%%%%%%%%%%%%%%%%%%%%%%%%%%%%%%%%%%%%%%%%

\section{Appendix}

\subsection{Two UFDs}

The purpose of this section is to show the following.

\begin{proposition}\label{UFD} The rings $\R [x,y]/(x^2+y^2+1)$ and $\R [x,y,z]/(x^2+y^2+z^2)$ are UFDs. 
\end{proposition}

\begin{proof} Let $F$ be a formally real field. Given nonzero $a\in F$, 
let $C'$ be the affine plane curve (over $F$) defined by $X^2+Y^2+a^2=0$. Then $C'$ is irreducible over $F$. 

Let $C$ be the projective extension of $C'$. Then $C$ is given by the equation
\[
X^2+Y^2+a^2T^2=X^2+Y^2+(aT)^2=0
\]
in projective coordinates $X,Y,T$. Since $a\ne 0$ and $F$ is formally real, $C$ has no $F$-rational points. By Samuel's Criterion (cited below), it follows that the ring
$F[X,Y]/(X^2+Y^2+a^2)$ is a UFD. This is true, in particular, if $F=\R$ and $a=1$.  

Likewise, if $F=\R (z)=\R ^{(1)}$ and $a=z$, then $F$ is a formally real field (see \cite{Freudenburg.17}, Lemma 10.10)
and the ring $\mathcal{B}=\R (z) [x,y]/(x^2+y^2+z^2)$ is a UFD. Define the subring 
\[
B=\R [z][x,y]/(x^2+y^2+z^2)
\]
and define the multiplicative subset $S\subset B$ by $S=\R [z]\setminus\{ 0\}$. 
If $P\subset S$ denotes the set of linear and irreducible quadratic polynomials in $\R [z]=\R^{[1]}$, then $S$ is generated by $P$. It is easy to check that, for each $p\in P$, the ring
$B/pB$ is a domain. Therefore, elements of $P$ are prime in $B$. Since $\mathcal{B}=S^{-1}B$, it follows by Nagata's Criterion (cited below) that $B$ is a UFD.
\end{proof}

\begin{theorem}\label{Samuel} {\rm (Samuel \cite{Samuel.61}, Thm. 5.1 (a))}  Let $C'$ be an irreducible conic in the affine plane, defined over a field $k$, 
let $A$ be its affine coordinate ring over $k$, and let $C$ be the projective extension of $C'$. If $C$ has no $k$-rational point, then $A$ is a UFD. 
\end{theorem}

\begin{theorem}\label{Nagata} {\rm (Nagata \cite{Nagata.57})}  Let $A$ be an integral domain and $S\subset A$ the multiplicative system generated by any family of prime elements of $A$.  If $S^{-1}A$ is a UFD, then $A$ is a UFD. 
\end{theorem}

%%%%%%%%%%%%%%%%%%%%%%%%%%%%%%%%%%%%%%%%%%%%%%%%%%%%%%%%%%%%%%%%%%%%%%%%%%%%%%%%

\subsection{Remarks and Questions}

\begin{enumerate}

\item The elliptic curve with $j$-invariant 1728 has two real forms $C_1$ and $C_2$, given by $y^2=x(x^2-1)$ and $y^2=x(x^2+1)$. What are 
$\Pi (C_1)$ and $\Pi (C_2)$?
\medskip

\item Since $\Sp^2$ is factorial, {\it Thm.\,\ref{A=UFD3}} shows that $\Pi (\Sp^2)$ is torsion free. But since its complexification is not factorial, we do not know if $\Pi (\Sp^2)$ is a free abelian group. To explore this problem, one can consider whether certain subgroups are free abelian, for example, the subgroup generated by $[c_1X+c_2Y+c_3Z+c_4]$, where $X^2+Y^2+Z^2=1$ and $c_j\in\C$. 
The inherent difficulty is seen in the group relation:
\[
[X+Y+iZ+1][X-Y+iZ+1]=[2X+2][X+iZ]=[X+iZ]
\]
\item Since sets of the same cardinality generate free abelian groups which are isomorphic as groups, it is natural to look for a stronger notion of isomorphism for polar groups, one which encodes some of the underlying topology and geometry. For example, in the case of the polynomial ring
$\C [x]$, the set of prime divisors $x-z$, $z\in\C$, forms an algebraic variety $V$ which is isomorphic to $\C$. In this case, $\mathcal{F}(V)$ can be given the structure of a topological group by declaring that, if $\mathcal{U}$ is a system of basic open sets on $V$, then the subgroups $\{ \mathcal{F}(U) : U\in\mathcal{U}\}$ form a system of basic open sets on $\mathcal{F}(V)$. 
The polar group $\Pi (A)$ is a free quotient of $\mathcal{F}(V)$ and takes the quotient topology.
This idea can be extended to case where $B$ is a UFD and the set of prime elements of $B$ has the structure of an infinite-dimensional $\C$-variety (ind-variety). 
Clearly, these ideas require further development. 

\end{enumerate}

\bibliographystyle{amsplain}

\noindent \address{Department of Mathematics\\
Western Michigan University\\
Kalamazoo, Michigan 49008}\\
\email{gene.freudenburg@wmich.edu}
\bigskip

\end{document}